\documentclass{amsart}

\usepackage{hyperref}
\usepackage{amsmath,array,amssymb,color}

\usepackage[all]{xy}
\usepackage[a4paper]{geometry}

\newcommand{\lcqKnospace}{\textsc{lcqK}}
\newcommand{\lcqK}{\lcqKnospace\ }
\newcommand{\lcpnospace}{\textsc{lcp}}
\newcommand{\lcp}{\lcpnospace\ }

\newcommand{\Id}{\mathop{\mathrm{Id}}}
\newcommand{\Cl}[1]{\mathrm{Cl}(#1)}
\newcommand{\Isom}[1]{\ensuremath{\text{\upshape\rmfamily Isom}(#1)}}

\newcommand{\Spin}[1]{\ensuremath{\text{\upshape\rmfamily Spin}(#1)}}
\newcommand{\Spindelta}[1]{\ensuremath{\text{\upshape\rmfamily Spin}(#1)_\Delta}}

\newcommand{\SO}[1]{\ensuremath{\text{\upshape\rmfamily SO}(#1)}}
\newcommand{\U}[1]{\ensuremath{\text{\upshape\rmfamily U}(#1)}}
\newcommand{\Gtwo}{\ensuremath{\text{\upshape\rmfamily G}_2}}
\newcommand{\SU}[1]{\ensuremath{\text{\upshape\rmfamily SU}(#1)}}

\newcommand{\GL}[1]{\ensuremath{\text{\upshape\rmfamily GL}(#1)}}
\newcommand{\End}[1]{\ensuremath{\text{\upshape\rmfamily End}(#1)}}
\newcommand{\Sp}[1]{\ensuremath{\text{\upshape\rmfamily Sp}(#1)}}
\newcommand{\Spdiag}[1]{\ensuremath{\text{\upshape\rmfamily Sp}(#1)_\Delta}}
\newcommand{\ug}{\;\shortstack{{\tiny\upshape def}\\=}\;}
\newcommand{\snform}{\ensuremath{\Phi}}
\newcommand{\spinform}[1]{\ensuremath{\Phi_{\Spin{#1}}}}

\newcommand{\cyclic}[1]{\ensuremath{\text{\upshape\rmfamily C}_{#1}}}
\newcommand{\bid}[1]{\ensuremath{\text{\upshape\rmfamily D}_{#1}}}
\newcommand{\bit}{\ensuremath{\text{\upshape\rmfamily T}}}
\newcommand{\bio}{\ensuremath{\text{\upshape\rmfamily O}}}
\newcommand{\bii}{\ensuremath{\text{\upshape\rmfamily I}}}
\newcommand{\liespin}[1]{\mathop{\mathfrak{spin}}(#1)}
\newcommand{\liesp}[1]{\mathop{\mathfrak{sp}}(#1)}
\newcommand{\lieso}[1]{\mathop{\mathfrak{so}}(#1)}
\newcommand{\lie}[1]{\mathop{\mathfrak{#1}}}

\newcommand{\st}{\vert}

\newcommand{\HH}{\mathbb{H}}
\newcommand{\RR}{\mathbb{R}}   
\newcommand{\ZZ}{\mathbb{Z}} 
\newcommand{\NN}{\mathbb{N}}

\newcommand{\OO}{\mathbb{O}}
\newcommand{\CP}[1]{\mathbb{C}P^{#1}}
\newcommand{\OP}[1]{\mathbb{O}P^{#1}}
\newcommand{\OH}[1]{\mathbb{O}H^{#1}}
\newcommand{\HP}[1]{\mathbb{H}P^{#1}}
\newcommand{\II}{\textrm{Im}}

\newcommand{\I}{\mathcal{I}} 

\numberwithin{equation}{section}

\theoremstyle{plain}
\newtheorem{te}{Theorem}[section]
\newtheorem*{te*}{Theorem}
\newtheorem*{teA}{\hyperlink{pr:proofA}{Theorem A}}
\newtheorem*{teB}{\hyperlink{pr:proofB}{Theorem B}}
\newtheorem*{teC}{\hyperlink{pr:proofC}{Theorem C}}
\newtheorem{pr}[te]{Proposition}
\newtheorem{co}[te]{Corollary}
\newtheorem{lm}[te]{Lemma}
\newtheorem*{triality}{\textbf{\upshape\rmfamily Triality Principle}}

\theoremstyle{remark}
\newtheorem{re}[te]{Remark}

\theoremstyle{definition}
\newtheorem{de}[te]{Definition}
\newtheorem{ex}[te]{Example}

\newenvironment{proofA}{\textbf{Proof of \hyperlink{te:A}{Theorem A}.}}{\hfill\qed}
\newenvironment{proofB}{\textbf{Proof of \hyperlink{te:B}{Theorem B}.}}{\hfill\qed}
\newenvironment{proofC}{\textbf{Proof of \hyperlink{te:C}{Theorem C}.}}{\hfill\qed}

\begin{document}

\title{\text{\upshape\rmfamily Spin}\textbf{(9)} geometry of the octonionic Hopf fibration}

\date{}

\author{Liviu Ornea}
\address{University of Bucharest\\ Faculty of Mathematics and Informatics\\ 14 Academiei str.\\ 70109 Bucharest, Romania\\
and Institute of Mathematics ``Simion Stoilow'' of the Romanian Academy\\ 21 Calea Grivitei Str.\\ 010702 Bucharest, Romania}
\email{\href{mailto:lornea@fmi.unibuc.ro}{lornea@fmi.unibuc.ro}}

\author{Maurizio Parton}
\address{Universit\`a di Chieti-Pescara\\ Dipartimento di Economia\\ Viale Pindaro 87\\ I-65127 Pescara, Italy}
\email{\href{mailto:parton@unich.it}{parton@unich.it}}

\author{Paolo Piccinni}
\address{Dipartimento di Matematica\\ Sapienza - Universit\`a di Roma\\
Piazzale Aldo Moro 2\\ I-00185, Roma, Italy}
\email{\href{mailto:piccinni@mat.uniroma1.it}{piccinni@mat.uniroma1.it}}

\author{Victor Vuletescu}
\address{University of Bucharest\\ Faculty of Mathematics and Informatics\\ 14 Academiei str.\\ 70109 Bucharest, Romania}
\email{\href{mailto:vuli@fmi.unibuc.ro}{vuli@fmi.unibuc.ro}}

\thanks{L.~O.\ and V.~V.\ were partially supported by CNCS UEFISCDI, project number
PN-II-ID-PCE-2011-3-0118 and by the INdAM-GNSAGA visiting program.}
\thanks{M.~P.\ and P.~P.\ were supported by the MIUR under the 2010-11 PRIN Project ``Variet\`a reali e complesse: geometria, topologia e analisi armonica''.}

\maketitle

\begin{abstract}
We deal with Riemannian properties of the octonionic Hopf fibration $S^{15}\rightarrow S^8$, in terms of the structure given by its symmetry group $\Spin{9}$. In particular, we show that any vertical vector field has at least one zero, thus reproving the non-existence of $S^1$ subfibrations. We then discuss $\Spin{9}$-structures from a conformal viewpoint and determine the structure of compact locally conformally parallel $\Spin{9}$-manifolds. Eventually, we give a list of examples of locally conformally parallel $\Spin{9}$-manifolds.
\end{abstract}

\section{Introduction}\label{sec:intro}

There are some features that distinguish $S^{15}$ among spheres of arbitrary dimension. For example, $S^{15}$ is the only sphere that admits three homogeneous Einstein metrics (see~\cite{ZilHEM}), and the only one that appears as regular orbit in three cohomogeneity one actions on projective spaces, namely of $\SU{8}$, $\Sp{4}$ and $\Spin{9}$ on $\CP{8}$, $\HP{4}$ and $\OP{2}$ respectively (see~\cite{KolCHC}). Moreover, according to a famous problem of vector fields on spheres, $S^{15}$ is the lowest dimensional sphere with more than 7 linearly independent vector fields (cf.~for example~\cite{HusFiB}). Finally, it has been shown that the Killing superalgebra of $S^{15}$ is isomorphic to the exceptional compact real Lie algebra $\lie{e}_8$ (see~\cite{FigGCE}).

All of these features can somehow be traced back to the transitive action of the subgroup $\Spin{9}\subset\SO{16}$ on the octonionic Hopf fibration $S^{15}\rightarrow S^{8}$. This latter has a quite exceptional character: it does not admit any $S^1$-subfibration (see~\cite{LoVHFS}), and there is no Hopf fibration over the Cayley projective plane $\OP{2}$, although its volume is the quotient of those of the spheres $S^{23}$ and $S^7$, natural candidates to its possible total space and fiber (cf.~\cite[page 8]{BerCCP}).

The mentioned characterizations of $S^{15}$ and the role of $\Spin{9}$ in 16-dimensional Riemannian geometry have been a first motivation for the present paper.

In this respect, a first result we get is the following:

\begin{teA}\hypertarget{te:A}{}
Any global vector field on $S^{15}$ which is tangent to the fibers of the octonionic Hopf fibration $S^{15} \rightarrow S^{8}$ has at least one zero.
\end{teA}

Note that the non-existence of $S^1$-subfibrations follows (cf.\ results obtained in~\cite{LoVHFS} and Corollary~\ref{cor}).

A second motivation for this paper is to complete a general scheme of description for metrics which are locally conformally parallel with respect to the $G$-structures that refer to Riemannian holonomies.
We next recall this general scheme. We say that we have a {\em locally conformally parallel $G$-structure} on a manifold $M$ if one has  a Riemannian metric $g$ on $M,$ a covering ${\mathcal U}=\{U_\alpha\}_{\alpha \in A}$ of $M,$ and for each $\alpha \in A$ a metric $g_\alpha$ defined on $U_\alpha$ which has holonomy contained in $G$ such that the restriction of $g$ to each $U_\alpha$ is conformal to $g_\alpha$:
$$g_{\vert U_\alpha}=e^{f_\alpha}g_\alpha$$
for some smooth map $f_\alpha$ defined on $U_\alpha.$

Some of the possible  cases here are:

\begin{itemize}
\item $G=\U{n}$, where we have the {\em locally conformally K\"ahler metrics};

\item $G=\Sp{n}\cdot \Sp{1}$, yielding the {\em locally conformally quaternion K\"ahler metrics};

\item $G=\Spin{9},$ which is the case we are dealing with.
\end{itemize}

In any of the cases above, one can show that for each overlapping $U_\alpha, U_\beta$ the functions $f_\alpha, f_\beta$ differ by a constant: 
$$f_\alpha-f_\beta=c_{\alpha, \beta}\,\, \text{on}\,\, U_\alpha\cap U_\beta.$$
This implies that  $df_\alpha=df_\beta$ on $U_\alpha \cap U_\beta\neq \emptyset$, hence defining a global, closed 1-form, usually denoted by $\theta$ and called the \emph{Lee form}.
Its metric dual with respect to $g$ is denoted by $B$:
$$ B=\theta^{\sharp}$$ 
and is called the \emph{Lee vector field.}
 
The  case $G=\U{n}$ is extensively studied: see for instance \cite{DrOLCK}. 

Choosing  $G$ to be  $\Sp{n}$ or  $\Sp{n}\cdot\Sp{1}$, we get  close relations to $3$-Sasakian geometry: see~\cite{OrPLCK} or the surveys~\cite{BoGTSM},~\cite{CaPEWG}. Finally, locally conformally parallel $\Gtwo$ and $\Spin{7}$-structures have been studied in~\cite{IPPLCP}, and they relate to nearly parallel $\SU{3}$ and $\Gtwo$ geometries, respectively. 

In the case we deal with in this paper, it is a classical result by D.~Alekseevsky that holonomy $\Spin{9}$ is only possible on manifolds that are either flat or locally isometric to $\OP{2}$ or to the hyperbolic Cayley plane $\OH{2}$ (see~\cite{AleRSE} and~\cite{BrGRMH}). Still,  weakened holonomy conditions have been also considered.  In particular, the article~\cite{FriWSS} points out how, exactly like in the frameworks of structure groups $\U{n}$ and $\Gtwo$, one can obtain 16 classes of $\Spin{9}$-structures. 

One of these classes consists of structures \emph{of vectorial type} (see~\cite{AgFGSV} and~\cite[page~148]{FriWSS}); we show that this class fits into the locally conformally parallel scheme above (see Remark~\ref{re:victor}). 

Besides this Remark, our contribution to the completion of the above general scheme with the case $G=\Spin{9}$ consists in  the following Theorems.


\begin{teB}\hypertarget{te:B}{}
Let $M^{16}$ be a compact manifold equipped with a locally, non globally, conformally parallel $\Spin{9}$ metric $g$. Then:
\begin{enumerate}
\item\label{te:st1} The Riemannian universal covering $(\tilde{M},\tilde{g})$ of $M$ is conformally equivalent to the Euclidean space ${\mathbb R}^{16}\setminus \{0\}$, that is, the Riemannian cone over $S^{15}$, and
$M$ is finitely isometrically covered by $S^{15} \times \RR$.
\item\label{te:st2} $M$ is equipped with a  canonical $8$-dimensional foliation.
\item\label{te:st3} 
If all the leaves of $\mathcal F$ are compact, then M fibers over an orbifold $\mathcal O^8$ finitely covered by $S^8$ and all fibers are finitely covered by $S^7 \times S^1$.
\end{enumerate}
\end{teB}


\begin{teC}\hypertarget{te:C}{}
Let $(M,g)$ be a compact Riemannian manifold. Then $(M,g)$ is locally, non globally, conformally parallel $\Spin{9}$ if and only if the following three properties are satisfied:
\begin{enumerate}
\item\label{it:structure1} $M$ is the total space of a fiber bundle
\[
M\stackrel{\pi}{\longrightarrow} S^1_r
\]
where $\pi$ is a Riemannian submersion over the circle of a certain radius $r$.
\item\label{it:structure2} The fibers of $\pi$ are isometric to a 15-dimensional spherical space form $S^{15}/K$, where $K\subset\Spin{9}$.
\item\label{it:structure3} The structure group of $\pi$ is contained in the normalizer $N_{\Spin{9}}(K)$ of $K$ in $\Spin{9}$ (that is, the isometries of $S^{15}/K$ induced by $\Spin{9}$).
\end{enumerate}
\end{teC}

\section{Preliminaries}\label{preliminaries}

Let $\OO$ be the algebra of octonions. The multiplication of $x=h_1+h_2 e$, $x'=h'_1+h'_2e \in \OO$
is defined through the one in quaternions $\HH$ by the Cayley-Dickson process:
\begin{equation}\label{oct}
xx'=(h_1h'_1-\overline h'_2 h_2) + (h_2 \overline h'_1 + h'_2 h_1)e,
\end{equation}
where $\overline h'_1, \overline h'_2$ are the conjugates of $h'_1, h'_2 \in \HH$. The conjugation in $\OO$ is defined by $\overline{x}=\overline h_1-h_2e$
and relates with the non-commutativity in $\OO$ by $\overline{x x'}=\overline x'\overline x.$ The non-associativity of $\OO$ gives rise to the associator 
\[
[x,x',x''] = (xx')x''-x(x'x''),
\] 
that vanishes whenever two among $x,x',x''$ are equal or conjugate. For a survey on octonions and their applications in geometry, topology and mathematical physics, see~\cite{BaeOct}.

We recall in particular the decomposition of the real vector space $\OO^2$ into its \emph{octonionic lines}  
\[
l_m\ug\{(x,mx)\st x\in\OO\} \quad \text{or} \quad l_\infty\ug\{(0,x')\st x'\in\OO\},
\] 
that intersect each other only in $(0,0) \in \OO^2$ (cf.\ Section \ref{versus Hopf}). Here $m \in S^8 =  \OP{1} = \OO \cup \{\infty\}$ parametrizes the set of octonionic lines $l$, whose volume elements $\nu_{l} \in \Lambda^8 l$ allow to define the following \emph{canonical 8-form} on $\OO^2 = \RR^{16}$:
\[
\spinform{9}\ug \int_{\OP{1}}p_l^*\nu_l\,dl\in\Lambda^8(\RR^{16}), 
\]
where $p_l$ denotes the orthogonal projection $\OO^2 \rightarrow l$. This definition of $\spinform{9}$ is due to M.~Berger (cf.\ \cite{BerCCP}). The following statement motivates our choice of notation for the canonical 8-form:

\begin{pr}\label{de:spin9} \cite[Proposition 1.4 at page 170]{CorASH}
The subgroup of $\mathrm{GL}(16, \RR)$ preserving $\spinform{9}$ is the image of $\Spin{9}$ under its spin representation into $\RR^{16}$.
\hfill\qed
\end{pr}

As such, one can look at $\Spin{9}$ as a subgroup of $\SO{16}$. Accordingly, $\Spin{9}$-structures can be considered on 16-dimensional oriented Riemannian manifolds. The following definition collects different approaches that have been used (see \cite{CorASH}, \cite{FriWSS}, \cite{PaPSAC}):

\begin{de}\label{de:spin9structure}
Let $M$ be a 16-dimensional oriented Riemannian manifold. A \emph{$\Spin{9}$-structure} on $M$ is the datum of any of the following equivalent alternatives.
\begin{enumerate}
\item A rank $9$ vector subbundle $V^9\subset\End{TM}$, locally spanned by endomorphisms
\[
\{\I_\alpha\}_{\alpha=1,\dots 9}
\]
satisfying
\begin{equation}\label{top}
\I^2_\alpha = \Id, \qquad \I^*_\alpha = \I_\alpha, \quad \text{and} \quad \I_\alpha \I_\beta = - \I_\beta \I_\alpha \quad\text{for}\quad \alpha \neq \beta,
\end{equation}
where $\I^*_\alpha$ denotes the adjoint of $\I_\alpha$.
\item An 8-form $\spinform{9}\in\Lambda^8(M)$ which can be locally written as in~\cite[Table~B]{PaPSAC}, for a certain orthonormal local coframe $\{e^1,\dots,e^{16}\}$.
\item A reduction $\mathcal{R}$ of the principal bundle of orthonormal frames on $M$ from $\SO{16}$ to $\Spin{9}$.
\end{enumerate}
\hfill\qed
\end{de}

\begin{re}
From any of the Definitions~\ref{de:spin9structure}, it follows that admitting a $\Spin{9}$-structure depends only on the conformal class of $M$.
\hfill\qed
\end{re}

We describe now the rank 9 vector bundle of endomorphisms when $M$ is the model space $\RR^{16}$. Here $\I_1, \dots, \I_9$ can be chosen as generators of the Clifford algebra $\Cl{9}$, the endomorphisms' algebra of its $16$-dimensional real representation $\Delta_9 = \RR^{16} = \OO^2$.
Accordingly, unit vectors $v \in S^8 \subset \RR^9$ can be viewed, via the Clifford multiplication, as symmetric endomorphisms $v: \Delta_9 \rightarrow \Delta_9$. 

The explicit way to describe this action is by $v = u + r \in S^8$ ($u \in \OO$, $r \in \RR$, $u\overline u + r^2 =1$), acting on pairs $(x,x') \in \OO^2$ by
\begin{equation}\label{HarSpC}
\left(
\begin{array}{c}
x \\
x'
\end{array}
\right)
\longrightarrow
\left(
\begin{array}{cc}
r & R_{\overline u} \\ 
R_u & -r
\end{array}
\right)  
\left(
\begin{array}{c} 
x \\ 
x'
\end{array}
\right),
\end{equation}
where $R_u, R_{\overline u}$ denote the right multiplications by $u, \overline u$, respectively  (cf.~\cite[page 288]{HarSpC}).

A basis of the standard $\Spin{9}$-structure on $\OO^2 = \RR^{16}$ can be written by looking at the action~\eqref{HarSpC} and at the nine vectors 
\[
(0,1),(0,i),(0,j),(0,k),(0,e),(0,f),(0,g),(0,h)\quad\text{and}\quad(1,0) \in S^8 \subset \OO \times \RR = \RR^9.
\]
In this way, one gets the following symmetric endomorphisms:
\[
\I_1=\left(
\begin{array}{c|c}
0 & \Id\\
\hline
\Id & 0
\end{array}
\right),
\I_2=\left(
\begin{array}{c|c}
0 & -R_i\\
\hline
R_i & 0
\end{array}
\right),
\dots,
\I_8=\left(
\begin{array}{c|c}
0 & -R_h\\
\hline
R_h & 0
\end{array}
\right),
\I_9=\left(
\begin{array}{c|c}
\Id & 0\\
\hline
0 & -\Id
\end{array}
\right),
\]
where $R_i,\dots,R_h$ are the right multiplications by the 7 unit octonions $i,\dots,h$. The subgroup $\Spin{9} \subset \SO{16}$ is then characterized as preserving the 9-dimensional vector space
\begin{equation}\label{eq:V9}
V^9\ug<\I_1,\dots,\I_9>\subset\End{\RR^{16}}.
\end{equation}


\section{The quaternionic Hopf fibration}\label{Quaternionic Hopf}

It is useful to look at $\Spin{9}\subset \SO{16}$ as the octonionic analogue of the quaternionic group $\Sp{2}\cdot\Sp{1}\subset\SO{8}$. A simple aspect of the analogy is given by the symmetry groups of the Hopf fibrations $S^7\rightarrow S^4$ and $S^{15}\rightarrow S^8$, that are $\Sp{2}\cdot\Sp{1}\subset\SO{8}$ and $\Spin{9}\subset\SO{16}$, respectively (see~\cite[pages 183 and 190]{GWZGHF}).

In the symmetry group of the quaternionic Hopf fibration, the two factors $\Sp{2}$ and $\Sp{1}$ act on the basis $S^4$ on the left, and on the $S^3$ fibers on the right, respectively. This action is thus related with the reducibility of the Lie algebra $\liesp{2}\oplus\liesp{1}$ and with the associativity of quaternions. All of this fails for the octonionic Hopf fibration, due to the irreducibility of $\liespin{9}$ and to the non-associativity of octonions.

However, the approach to a $\Spin{9}$-structure on a 16-dimensional manifold $M$ through the vector bundle $V^9\subset\End{TM}$ admits a strict analogy for $\Sp{2}\cdot\Sp{1}$. The same formula \eqref{HarSpC} defines a similar action on the sphere $S^4$ in $\HH^2$, and this can be viewed as defining a $\Sp{2}\cdot\Sp{1}$-structure. An explicit description of a canonical basis $\I_1,\I_2,\I_3,\I_4,\I_5$ of sections of $V^5\subset\End{\HH^2}$ is given by the choices $(r,u)=(0,1),(0,i),(0,j),(0,k),(1,0)$ in equation~\eqref{HarSpC}, where now $u,x,x'\in\HH$, and thus $(r,u)\in S^4$ (cf.~\cite{PaPSAC}). 

The ten compositions $\I_\alpha\I_\beta$, for $\alpha < \beta$, yield complex structures on $\RR^8 = \HH^2$, and a basis of the Lie algebra $\liesp{2}$. 
In particular, the sum of squares of their K\"ahler forms $\omega_{\alpha\beta}$ gives (cf.~\cite[page 329]{PaPSAC}):
\[
\sum_{1\leq \alpha < \beta \leq 5} \omega^2_{\alpha\beta} = -2\Omega_L,
\]
where $\Omega_L$ is the left quaternionic 4-form in $\RR^8$, defined as usual by
\[
\Omega_L \ug \omega^2_{L_i} +  \omega^2_{L_j} + \omega^2_{L_k},
\]
in terms of the K\"ahler forms $\omega$ of the left multiplications $L_i,L_j$ and $L_k$.

Thus, on a Riemannian manifold $M^8$, the datum of a $\Sp{2}\cdot\Sp{1}$-structure can be given through two different approaches. One can simply fix the usual rank 3 vector subbundle $Q^3$ of skew-symmetric elements in $\End{TM}$, whose local generators can be denoted by $I,J,K$. In the model space $\RR^8$, the subgroup of rotations commuting with the standard complex structures $I,J,K$ is $\Sp{2} \subset \SO{8}$, and the second factor $\Sp{1}$ of the reduced structure group here works as the double covering of $\SO{3}$, allowing to change the admissible hypercomplex structure. 

Since both factors of $\Sp{2}\cdot \Sp{1}$ are double coverings of rotation groups - namely of $\SO{5}$ and $\SO{3}$, respectively - one can reverse the role of the two factors. Accordingly, one can follow a different approach to fix a $\Sp{2}\cdot \Sp{1}$ reduction of the structure group on a Riemannian $M^8$. This second approach is what can be called a \emph{quaternionic Hopf structure} (cf.~\cite[page 327]{PaPSAC}), and consists of a vector subbundle $V^5 \subset \End{TM}$ of symmetric elements, whose local bases of sections $\I_\alpha \in \Gamma (V^5)$ $(\alpha = 1, \dots , 5)$ satisfy relations~\eqref{top}. On the model space $\RR^8$, the subgroup of rotations commuting with the standard $\I_1, \dots , \I_5$ is the diagonal $\Sp{1}$ subgroup of $\SO{8}$, and now it is the left factor of $\Sp{2}\cdot \Sp{1}$ to allow admissible five dimensional rotations in the choice of bases of sections in $V^5$. As already recalled, the quaternionic 4-form of $\HH^2 \cong \RR^8$ can be easily written according to both the mentioned approaches.

In Section~\ref{se:lcp}, we will deal with locally conformally parallel $\Spin{9}$-structures. It will be useful to have in mind some known facts for their corresponding $8$-dimensional analogues, Riemannian manifolds $M^8$ whose metric is locally conformally related to metrics with holonomy $\Sp{2}\cdot\Sp{1}$. We rephrase here some of these facts in terms of the rank $5$ vector bundle $V^5 \subset \End{TM}$.

As mentioned in the Introduction, a quaternion Hermitian manifold $(M^8,g)$ is called \emph{locally conformally quaternion K\"ahler} (or \lcqKnospace, briefly) if $g_{\vert U} = e^{f_U} g'_U$ with local quaternion K\"ahler metrics $g'_U$, defined over open neighborhoods $U$ covering $M$. The \emph{Lee form} $\theta$, locally $\theta_{\vert U} = df_U$, allows to characterize globally the \lcqK condition (cf.~\cite[page 643]{OrPLCK}):
\[
d\Omega_L = \theta \wedge \Omega_L, \qquad d\theta=0.
\]
The Levi-Civita connections of local quaternion K\"ahler metrics $g'_U$ glue together to the \emph{Weyl connection} $D$, defined on tangent vector fields $X,Y$ as
\[
D_XY=\nabla_XY-\frac{1}{2}\{\theta(X)Y+\theta(Y)X-g(X,Y)B\},
\]
where $\nabla$ is the Levi-Civita connection of $g$ and $B=\theta^\sharp$ is the \emph{Lee vector field}. Then the \lcqK condition can be viewed as an example of \emph{Einstein-Weyl structure}, i.e.\ the datum of the conformal class $[g]$ of metrics together with the torsion-free connection $D$ satisfying the Einstein condition and preserving both the conformal class $[g]$ and the vector bundle $V^5 \rightarrow M^8$, that is, satisfying $Dg = \theta \otimes g$ and $DV^5 \subset V^5$.

Abundant examples exist in the subclass of $8$-dimensional compact locally conformally hyperk\"ahler manifolds: for instance any product $\mathcal{S}\times S^1$ of a compact 3-Sasakian 7-dimensional manifold $\mathcal{S}$ with a circle, where the former can be chosen having any second Betti number $b_2(\mathcal{S})$ (see~\cite{BGMCTS}). 

However,  \lcqK metrics on compact $M^8$ are either globally conformally quaternion
K\"ahler or locally conformally quaternion K\"ahler with the local quaternion K\"ahler
metrics of vanishing scalar curvature (\cite[page 645]{OrPLCK}), so that the locally quaternion K\"ahler metrics $g'_U$ are necessarily Ricci flat. 

Note that this does not imply the existence of a global hypercomplex structure on $M^8$, even on the open neighborhood where the local hyperk\"ahler metrics $g'_U$ are defined.

In the following, we see how a locally conformally quaternion K\"ahler manifold $M^8$ can be described by looking at the vector bundle $V^5 \rightarrow M^8$, and by using the vector fields $\I_1 B, \dots, \I_5 B$ on $M^8$.

\begin{lm}\label{le:lchKstructure}
Let $M^{8} $ be a compact manifold equipped with a locally, non globally, \lcqK metric $g$. Let $B$ be its Lee vector field and $V^5\subset\End{TM}$ the vector bundle defining the $\Sp{2}\cdot\Sp{1}$-structure, locally spanned by $\I_1,\dots,\I_5$. Then the local vector fields $\I_1B,\dots,\I_5B$ are orthonormal and $B$ belongs to their 5-dimensional distribution $VB$. The orthogonal complement $(VB)^\perp$ is integrable. 
\end{lm}

\begin{proof}
Consider on $M$ the distribution $\mathcal F$ spanned by the Lee vector field $B$ and its transformation under the (local) compatible almost complex structures. As already mentioned, the whole $\Sp{2}\cdot\Sp{1}$-structure can be given either by a rank 3 vector subbundle $Q^3$ of skew-symmetric elements in $\End{TM}$ (whose local generators are compatible almost complex structures usually denoted by $I,J,K$), or by a vector subbundle $V^5 \subset \End{TM}$ of symmetric elements, whose local generators we denote here by $\I_1,\dots \I_5$.

To prove the statement, there are now two possibilities. The first one is to refer to the work~\cite{OrPLCK}, and to rephrase the integrability of $\mathcal F$, a consequence of Frobenius Theorem in \cite[page 645]{OrPLCK}, in terms of the vector bundle $V^5$.
The geometric interplay between the foliation $\mathcal F$ and the distribution $VB$, locally spanned by the vector fields $\I_1B,\dots,\I_5B$, follows from a computation that can be performed in the model space $\RR^8$. This gives rise to the situation described in the statement. The same computation shows that none of the $\I_{\alpha}B$ is in general perpendicular to $B$, and that the orthogonal complement $(VB)^\perp$ is locally spanned by $IB,JB,KB$.

A second way to prove the statement is by a straightforward computation. This will be essentially done later in the proof of \hyperlink{te:B}{Theorem~B}, and more precisely of its statement (2). Although this latter refers to the $\Spin{9}$ case, the same computations, if limited to the choices $1, \dots , 5$, prove the present statement.
\end{proof}

The following Proposition gives now a more complete description of \lcqK manifolds in terms of the vector bundle $V^5$. Again, its proof follows from results in \cite{OrPLCK} (see in particular Theorem 3.8 at page 649).

\begin{pr}
Let $M^{8} $ be a compact manifold equipped with a locally, non globally, \lcqK metric $g$, with the same notation as in Lemma~\ref{le:lchKstructure}.
\begin{enumerate}
\item There exists a metric in the conformal class of $g$ whose Lee form $\theta$ is parallel.
\item On each integral manifold $N^7$ of $\ker(\theta)$, the distribution $(VB)^\perp$, orthogonal in $M$ to $VB$, is integrable and its leaves are 3-dimensional spherical space forms. The distribution on $M$ spanned by $(VB)^\perp$ and $B$ is the 4-dimensional vertical foliation $\mathcal{F}$, whose leaves are \lcqK (generally non primary) Hopf surfaces. 
\item The leaf space $M/\mathcal{F}$, when a manifold or an orbifold, carries a projected positive self-dual Einstein metric.
\end{enumerate}
\end{pr}

\section{$\Spin{9}$ and the octonionic Hopf fibration}\label{versus Hopf}

For any $(x,y)\in S^{15}\subset\OO^2=\RR^{16}$, we denote by 
\[
B\ug (x,y)\ug(x_1,\dots,x_8,y_1,\dots,y_8) 
\]
the (outward) unit normal vector field of $S^{15}$ in $\RR^{16}$. Here and in the following, we are identifying the tangent spaces $T_{(x,y)}(\RR^{16})$ with $\RR^{16}$.

Through the involutions $\I_1,\dots,\I_9$ one gets then the following sections of $T(\RR^{16})_{\vert_{S^{15}}}$ of length one:
\begin{equation}\label{spin9B}
\begin{split}
\I_1 B &=(y_1,y_2,y_3,y_4,y_5,y_6,y_7,y_8,x_1,x_2,x_3,x_4,x_5,x_6,x_7,x_8),\\
\I_2 B &=(y_2,-y_1,-y_4,y_3,-y_6,y_5,y_8,-y_7,-x_2,x_1,x_4,-x_3,x_6,-x_5,-x_8,x_7),\\
\I_3 B &=(y_3,y_4,-y_1,-y_2,-y_7,-y_8,y_5,y_6,-x_3,-x_4,x_1,x_2,x_7,x_8,-x_5,-x_6),\\
\I_4 B &=(y_4,-y_3,y_2,-y_1,-y_8,y_7,-y_6,y_5,-x_4,x_3,-x_2,x_1,x_8,-x_7,x_6,-x_5),\\
\I_5 B &=(y_5,y_6,y_7,y_8,-y_1,-y_2,-y_3,-y_4,-x_5,-x_6,-x_7,-x_8,x_1,x_2,x_3,x_4),\\
\I_6 B &=(y_6,-y_5,y_8,-y_7,y_2,-y_1,y_4,-y_3,-x_6,x_5,-x_8,x_7,-x_2,x_1,-x_4,x_3),\\
\I_7 B &=(y_7,-y_8,-y_5,y_6,y_3,-y_4,-y_1,y_2,-x_7,x_8,x_5,-x_6,-x_3,x_4,x_1,-x_2),\\
\I_8 B &=(y_8,y_7,-y_6,-y_5,y_4,y_3,-y_2,-y_1,-x_8,-x_7,x_6,x_5,-x_4,-x_3,x_2,x_1),\\
\I_9 B &=(x_1,x_2,x_3,x_4,x_5,x_6,x_7,x_8,-y_1,-y_2,-y_3,-y_4,-y_5,-y_6,-y_7,-y_8).\\
\end{split}
\end{equation}

As mentioned, $\Spin{9} \subset \SO{16}$ is the group of symmetries of the octonionic Hopf fibration. This latter is defined by looking at the decomposition of $\OO^2$ into the \emph{octonionic lines}  
\[
l_m\ug\{(x,mx)\st x\in\OO\} \quad \text{or} \quad l_\infty\ug\{(0,y)\st y\in\OO\},
\] 
mentioned in Section \ref{preliminaries}. One has to be careful that the octonionic line through $(0,0)$ and $(x,y) \in \OO^2$ \emph{is not} $\{(xo,yo)\st o\in\OO\}$. This latter in fact is not even an octonionic line, the correct line being instead $l_{yx^{-1}} = \{(o,(yx^{-1})o\st o\in\OO\}$ if $x \neq 0$, and $l_\infty$ if $x=0$. In this way the fibration 
\[
\OO^2 \setminus 0 \rightarrow S^8  = \{m \in \OO\} \cup \{\infty\}
\]
is obtained, with fibers $\OO \setminus 0$, and the intersection with the unit sphere $S^{15} \subset \OO^2$ provides the octonionic Hopf fibration 
\[
S^{15} \rightarrow S^8 , \qquad \text{or as homogeneous fibration}\qquad  \frac{\Spin{9}}{\Spin{7}}\stackrel{\frac{\Spin{8}}{\Spin{7}}}{\longrightarrow}\frac{\Spin{9}}{\Spin{8}}.
\]

Denote by $VB$ the 9-dimensional span of $\I_1 B,\dots,\I_9 B$:
\[
VB\ug<\I_1 B,\dots,\I_9 B>,
\]
and note that 9-planes of $VB$ are generally not tangent to $S^{15}$.

\hfill


\noindent\begin{proofA}\hypertarget{pr:proofA}{}
First, note that $VB$ is invariant under $\Spin{9}$: this is clear for the unit normal $B=(x,y)$, since $\Spin{9}\subset\SO{16}$, and on the other hand the nine endomorphisms $\I_\alpha$ are rotating under the $\Spin{9}$ action inside their vector space $V^9 \subset \End{\RR^{16}}$.


Next, $VB$ contains $B$. In fact:
\[
B = \lambda_1 \I_1 B + \lambda_2 \I_2 B + \dots + \lambda_8 \I_8 B + \lambda_9 \I_9 B,
\]
where the coefficients $\lambda_\alpha$ can be computed from \ref{spin9B}  in terms of the inner products (here all the arrows denote vectors in $\RR^8$) 
\[
\vec x =(x_1,\dots, x_8), \; \vec y=(y_1, \dots, y_8) \in \RR^8
\]
and of the right translations $R_i, \dots, R_h$ as follows:
\[
\lambda_1 = 2 \vec x \cdot \vec y, \; \lambda_2 =-2 \vec x \cdot \vec{R_i y}, \; \dots, \; \lambda_8 =-2 \vec x \cdot \vec{R_h y}, \; \lambda_9 = \vert \vec x \vert^2 - \vert \vec y \vert^2.
\]


In particular, at points with $\vec x = \vec 0$, that is on the octonionic line $l_\infty$, the vector fields $\I_1 B, \dots, \I_9 B$ are orthogonal to the unit sphere $S_\infty^7 \subset l_\infty$. This latter is the fiber of the Hopf fibration $S^{15} \rightarrow S^8$ over the north pole $(0, \dots, 0,1) \in S^8$, and the mentioned orthogonality of this fiber $S^7$ is immediate from \ref{spin9B} for $\I_1 B, \dots, \I_8 B.$ 
Also, for these points, we have  $\I_9 B=B$, so  $\I_9B$ is ortohogonal to $S_\infty^7$. Now, the invariance under $\Spin{9}$ of the octonionic Hopf fibration shows that all its fibers are characterized as orthogonal in $\RR^{16}$ to the vector fields $\I_1 B, \dots, \I_9 B$. 


Now, assume that $X$ is a vertical vector field of $S^{15} \rightarrow S^8$. By the previous characterization we have the following orthogonality relations in $\RR^{16}$:
\[
\langle X, \I_\alpha B \rangle =0, \qquad\text{for }\alpha=1, \dots, 9,
\]
and it follows that  $\langle \I_\alpha X, B \rangle=0$. But from the definition of a $\Spin{9}$-structure we see that if $\alpha \neq \beta$, then $\langle \I_\alpha X, \I_\beta X \rangle =0$. Thus, if $X$ is a nowhere zero vertical vector field, we would obtain in this way 9 pairwise orthogonal vector fields $\I_1 X, \dots, \I_9 X$, all tangent to $S^{15}$. But $S^{15}$ is known to admit at most 8 linearly independent vector fields by the classical Hurwitz-Radon-Adams result (see for example~\cite{HusFiB} or ~\cite{PaPSWM}). Thus $X$ cannot be vertical and nowhere zero, and \hyperlink{te:A}{Theorem~A} is proved.
\end{proofA}
 
One gets as a consequence the following alternative proof of a result  in~\cite{LoVHFS}:

\begin{co} \label{cor}
The octonionic Hopf fibration $S^{15} \rightarrow S^8$ does not admit any $S^1$ subfibration.
\end{co}

\begin{proof}
In fact, any $S^1$ subfibration would give rise to a real line subbundle $L\subset T_{\text{vert}}(S^{15})$ of the vertical subbundle of $T(S^{15})$. Such line bundle $L$ is necessarily trivial, due to the vanishing of its first Stiefel-Whitney class $w_1(L)\in H^1(S^{15};\ZZ_2)=0$. It follows that $L$ would admit a nowhere zero section, thus a global vertical nowhere zero vector field.
\end{proof}

\section{Locally conformally parallel $\Spin{9}$ manifolds}\label{se:lcp}

\begin{de}\label{de:lcp}
A Riemannian manifold $(M^{16},g)$ is \emph{locally conformally parallel} $\Spin{9}$ (\lcpnospace, briefly) if over open neighbourhoods
$\{U\}$ covering $M$ the restriction $g_{\vert U}$ of the metric $g$ is conformal to a (local) metric $g'_{U}$ having holonomy contained in $\Spin{9}$.
\hfill\qed
\end{de}

The conformality relations $g_{\vert U}=e^{f_U}g'_U$ give rise to a \emph{Lee form} $\theta$, locally defined as $\theta_{\vert U} \ug df_U$.
Next, recall that $\Spin{9}$ is characterized as the subgroup of $\GL{16, \RR}$ that preserves the 8-form $\spinform{9}$ (cf.\ the already quoted~\cite[page 170]{CorASH}). Thus, a $\Spin{9}$-structure on $M^{16}$ is equivalent to the datum of a $\Phi\in\Lambda^8(M),$ which can be locally written as in~\cite[Table~B]{PaPSAC} and, under the \lcp hypothesis, on each $U$ the metric $g'_U$ defines a similar $8$-form $\snform'_U$ parallel with respect to the Levi-Civita connection of $g'_U$. It follows that the restriction of $\Phi$ to $U$ satisfies
\[
\Phi_{\vert_{U}} = e^{4f_U} \snform'_U,
\]
henceforth one has

\[
d\Phi = \theta \wedge \Phi.
\]

Moreover, the Levi-Civita connections of the local parallel metrics $g'_{U}$ glue together to the global Weyl connection $D$ on $M$: 
\[
D_XY=\nabla_XY-\frac{1}{2} \{\theta(X)Y+\theta(Y)X-g(X,Y)B\},
\]
where $\nabla$ is the Levi-Civita connection of $g$. Recall that, since the metrics $g'_{U}$ are assumed to have holonomy contained in $\Spin{9}$, they are Einstein metrics. Thus the conditions $DV \subset V$, $Dg=\theta \otimes g$, $d\theta =0$ and $g'_{U}$ Einstein, insures that the conformal class $[g]$ defines a closed Einstein-Weyl manifold $(M, [g], D)$.

We will next give the

\hfill

\noindent\begin{proofB}\hypertarget{pr:proofB}{}
First, recall that $g$ defines a closed Einstein-Weyl structure on a compact manifold, with Lee form $\theta$ non exact (but closed). Then the following properties hold (cf.~\cite[page 10, Theorem 3]{GauSWE}): (a) its Weyl connection $D$ is Ricci-flat; (b) one can choose, in the conformal class $[g]$, a metric $g_0$ (unique up to homotheties) such that its Lee form $\theta_0$ is parallel with respect to the Levi-Civita connection $\nabla^{g_0}$ of $g_0$. 

Thus, to prove our statement we can assume, without loss of generality, that the Lee form $\theta$ of $g$ is parallel with respect to the Levi Civita connection $\nabla^{g}$.  Henceforth also its Lee field $B=\theta^\sharp$ is parallel and as a consequence, by de Rham decomposition theorem, the universal covering $(\tilde{M},\tilde{g})$ is reducible: 
\[
(\tilde{M},\tilde{g})=(\RR,dt^2)\times(\tilde{N},g_{\tilde N}),\quad\tilde{N}\text{ complete and simply connected}.
\]
With respect to this decomposition we have that the pull-back of $ \theta$ is $\tilde{\theta}=dt$. 
The diffeomorphism $\RR\times \tilde{N}\rightarrow \RR^+\times \tilde{N}$ given by
\[
(t, x)\mapsto (s=e^t, x)
\]
shows that  $(\tilde{M},\tilde{g})$ is globally conformal, with conformality factor $\frac{1}{s^2}$, to the so-called {\em metric cone} $$C(\tilde{N})=(\tilde{M},ds^2+s^2g_{\tilde N}).$$  

Using  the classical D.\ Alekseevsky theorem (\cite[Corollary 1 at page 98]{AleRSE}) we see that the Ricci-flatness of the local metrics (as mentioned, consequence of their holonomy contained in $\Spin{9}$ and of the Theorem of Gauduchon on closed Weyl structures), insures their flatness, so that the cone $C(\tilde{N})$ is flat. We can use then the relation between the curvature operator $R$ of the warped product $C(\tilde{N})=\RR^+\times_{s^2} \tilde{N}$ and the curvature operator $R^{\tilde{N}}$ of its fiber $\tilde{N}$:
\[
0=R_{V W} Z = R^{\tilde{N}}_{V W} Z - \frac{4}{s^2} (g_{\tilde N} (V,Z) W - g_{\tilde N} (W,Z) V)
\]
(see for example~\cite[page 210]{ONeSRG}) to recognize that $\tilde N$, being complete, is the sphere  $S^{15}$.
All of this insures that the universal covering of $M$ is conformally equivalent to the cone $C(S^{15})$ and, since the Lee vector field $B$ is parallel, that $M$ is locally isometric, up to homotheties, to $S^{15}\times\RR$.
This proves statement~\ref{te:st1}.

We now prove statement~\ref{te:st2}.
Denote by $\Theta$ the codimension 1 foliation on $M$ defined by the equation $\theta =0$, with $\theta =B^{\sharp}$, and note that the parallelism of $\theta$ insures that $\Theta$ is a totally geodesic foliation, so that the Levi-Civita connection on any leaf $T=T^{15}$ is just the restriction of $\nabla^{g}$. Next, consider the vector bundle $V=V^9\subset\End{TM}$ given by the $\Spin{9}$-structure, locally spanned by $\I_1,\dots,\I_9$, and the corresponding distribution
\[
VB\ug<\I_1 B,\dots,\I_9 B>\subset T(M)
\]
generated by the orthonormal vector fields $\I_1 B,\dots,\I_9 B$. Then $VB$ contains the Lee vector field $B$, as seen in the \hyperlink{pr:proofA}{proof of Theorem~A}.

We now show that the 8-dimensional distribution
\[
\mathcal F\ug<\I_1 B,\dots,\I_9 B>^\perp \oplus <B>=(VB)^\perp\oplus <B>
\]
is integrable.


First, let $X,Y \in (VB)^\perp$, so that $g(X, \I_\alpha B)=g(Y, \I_\alpha B)=0$ for $\alpha=1, \dots, 9$. Then, in terms of the Weyl connection $D$, we have
\[
g([X,Y], \I_\alpha B) = g(\I_\alpha (D_X Y),B)-g(\I_\alpha (D_Y X), B).
\]
Recall now that $DV \subset V$ gives rise to 1-forms $a_{\alpha \beta}$ such that $D \I_\alpha = \sum a_{\alpha \beta} \otimes \I_\beta$. It follows:
\[
g(\I_\alpha (D_X Y),B)=g( D_X (\I_\alpha Y),B)-g((D_X( \I_\alpha)  Y),B) = \]\[=g( D_X (\I_\alpha Y),B) - \sum a_{\alpha \beta} (X) g(\I_\beta Y, B),
\]
and since $g(\I_\beta Y, B)=0$, we obtain
\[
g(\I_\alpha (D_X Y),B) = g( D_X( \I_\alpha Y),B).
\]
On the other hand, since $\nabla B=0$ we have also $D_Y B = -\frac{1}{2}Y$ (cf.~\cite[page 37]{DrOLCK}). Thus, by applying $D$ to the identity $g(\I_\alpha X,B)=0$, we obtain
\[
0= Y(g(\I_\alpha X,B))=g(D_Y (\I_\alpha X),B) + g(\I_\alpha X,D_Y B)+\theta(Y)g(\I_\alpha X,B),
\]
so that
\[
g(D_Y(\I_\alpha X),B) = \frac{1}{2}g(\I_\alpha X,Y).
\]

All of this gives
\[
g([X,Y],\I_\alpha B) = g(D_X( \I_\alpha Y),B)-g(D_Y(\I_\alpha X), B) = \frac{1}{2}g(\I_\alpha Y,X) - \frac{1}{2}g(\I_\alpha X,Y)=0,
\]
so we get that $[X,Y] \in (VB)^\perp$.

Now, to obtain the integrability of $\mathcal F$, we must further check that for $X \in (VB)^\perp$ the bracket $[X,B] = D_X B-D_B X = -X -D_B X$ belongs to $\mathcal F$. In fact $D_B X= \nabla_B X - \frac{1}{2}X \in (VB)^\perp$, and $\nabla_B X  \in (VB)^\perp$ is a consequence of $g(X,B)=g(X, \I_\alpha B)=0$. This ends the proof of statement~\ref{te:st2}.

As for statement~\ref{te:st3}, one can use the same argument as in~\cite[Theorem 2.1]{OrPLCK}  to show that $\mathcal F$ is a Riemannian totally geodesic foliation and that the leaf space, when a manifold or an orbifold, carries a metric of spherical space form type.
\end{proofB}

\hfill

\noindent\begin{proofC}\hypertarget{pr:proofC}{}
Our arguments will follow basically the same ideas as in 
\cite{OrVSTC},
and we first show that the locally conformally parallel $\Spin{9}$ condition implies on compact manifolds the structure described by properties \eqref{it:structure1}, \eqref{it:structure2} and~\eqref{it:structure3}. 

In fact, if $(M,g)$ is compact and locally, non globally, conformally parallel $\Spin{9}$, recall from the proof of \hyperlink{pr:proofB}{Theorem~B} that its universal covering $(\tilde M,\tilde g)$ is conformally equivalent to the metric cone $C(S^{15})=\RR^{16}\setminus 0$ with conformal factor $\frac{1}{s^2}=e^{-2t}$, that is, the cone metric $g_{\text{cone}}$ is given by
\[
g_{\text{cone}}=e^{-2t}\tilde g.
\]
Any $\gamma\in\pi_1(M)$ can be thought as a map $\gamma:\tilde{M}\rightarrow\tilde{M}$ preserving $\tilde{g}$, and we get:
\begin{equation*}
\gamma^*(g_{\text{cone}})=\gamma^*(e^{-2t}\tilde{g})=(e^{-2t}\circ\gamma)\gamma^*(\tilde{g})=(e^{-2t}\circ\gamma)\tilde{g}=(e^{-2t}\circ\gamma)e^{2t}g_{\text{cone}},
\end{equation*}
showing that $\pi_1$ acts by conformal maps. 
Moreover, taking differentials of
\[
\gamma^*(\spinform{9})=(e^{-2t}\circ\gamma)e^{2t}\spinform{9}
\]
and using $d\spinform{9}=0$, we see that  $\pi_1(M)$ acts by homotheties.

Indeed, the homothety factor $\rho(\gamma)$ of $\gamma\in\pi_1(M)$ defines a homomorphism $\rho:\pi_1(M)\to\RR^+$, whose image is a finitely generated subgroup of $\RR^+$, thus isomorphic to $\ZZ^n$ for some $n\in\NN$. The locally conformal flatness of the metric allows to apply the arguments used to prove~\cite[Corollary 4.7]{GOPRVS}, and to see that the image of $\rho$ is isomorphic to $\ZZ$.

Next, notice that $K\ug\ker\rho$ consists of isometries of $C(S^{15})$ that leave the form $\spinform{9}$ invariant, so that in particular $K\subset\Spin{9}$.
Moreover, any isometry of $C(S^{15})$ induces the identity map on the $\RR^+$-component (see again~\cite[Theorem~5.1]{GOPRVS}), and it leaves the fibers of the projection $C(S^{15}) \rightarrow \RR^+$ invariant.
Since $S^{15}$ is compact and $\pi_1(M)$ acts properly discontinously and freely on $C(S^{15})$, $K$ is finite and without fixed points on $S^{15}$. It follows: 
\begin{equation*}
C(S^{15})/K=C\left(S^{15}/K\right).
\end{equation*}

Consider now a homothety $\gamma\in \pi_1(M)$ such that $h\ug\rho(\gamma)\in\RR^+$ generates $\II(\rho)$. Then $\gamma$ is a homotethy on $C(\frac{S^{15}}{K})$, and
\begin{equation}\label{eq:psiiso}
\gamma(s,x)=(h\cdot s,\psi(x)),\qquad\text{for }x\in\frac{S^{15}}{K},s\in\RR^+\text{ and }\psi\in\Isom{\frac{S^{15}}{K}}.
\end{equation}
Thus, for any $n\in\ZZ$ we have:
\begin{equation}\label{eq:psi}
\gamma^n(s,x)=(h^n\cdot s,\psi^n(x)).
\end{equation}

Consider the projection $\text{pr}:C(\frac{S^{15}}{K})\rightarrow\RR^+$ on the first factor of the cone. Then formula~\eqref{eq:psi} shows that $\text{pr}$ is equivariant with respect to the actions of $<\gamma>=\ZZ$ on $C(\frac{S^{15}}{K})$ and of $n\in\ZZ$ on $s\in\RR^+$, given by $h^n\cdot s$. The induced map
\begin{equation}\label{eq:structure}
M=\frac{C(\frac{S^{15}}{K})}{<\gamma>}\stackrel{\pi}{\longrightarrow}\frac{\RR^+}{\ZZ}=S^1
\end{equation}
is, up to rescaling the metric on $S^1$, the map in~\eqref{it:structure1} in the statement. Then~\eqref{it:structure2} follows. As for~\eqref{it:structure3}, observe that $\psi$ in formula~\eqref{eq:psiiso} comes from an element of $\SO{16}$ which preserves $\spinform{9}$.

To show that \eqref{it:structure1}, \eqref{it:structure2} and~\eqref{it:structure3} are necessary conditions for $M$ to be locally conformally parallel $\Spin{9}$, we use a topological argument. Assume that $(M,g)$ is a compact Riemannian manifold satisfying~\eqref{it:structure1}, \eqref{it:structure2} and~\eqref{it:structure3} in \hyperlink{te:C}{Theorem~C}, and consider two open sets $U_1$ and $U_2$ covering $S^1$. Then the definition of fibre bundle implies that $M$ can be recovered by glueing together $U_1\times (S^{15}/K)$ and $U_2\times (S^{15}/K)$ by a transition function $\psi_\pi:S^{15}/K\to S^{15}/K$. This transition function depends on $\pi$, and is usually called the \emph{clutching function} of the bundle. Moreover, \eqref{it:structure3} implies that $\psi\in N_{\Spin{9}}(K)$ is an isometry of $S^{15}/K$. 
Now choose $h\in\RR^+$, and use $\psi_\pi$, $h$ to define a homothety $\gamma_\pi$ on $C(S^{15}/K)$ as in formula~\eqref{eq:psiiso}:
\begin{equation*}
\gamma_\pi(s,x)\ug(h\cdot s,\psi_\pi(x)),\qquad\text{for }x\in\frac{S^{15}}{K}\text{ and }s\in\RR^+.
\end{equation*}
Then, let $M_\pi$ be the locally conformally parallel $\Spin{9}$ manifold
\[
M_\pi\ug\frac{C(\frac{S^{15}}{K})}{<\gamma_\pi>}.
\]
Since we already proved the sufficiency of conditions in \hyperlink{te:C}{Theorem~C}, we know that $M_\pi$ is itself a fiber bundle over $S^1$ with the same clutching map $\psi_\pi:S^{15}/K\to S^{15}/K$. Recall on the other hand that for any Lie group $G$, the equivalence classes of principal $G$ bundles over $S^n$ is in natural bijection with the homotopy group $\pi_{n-1}(G)$ (\cite[Theorem 18.5, page 99]{SteTFB}). 
Thus $M$ and $M_\pi$ are isomorphic as fiber bundles over $S^1$, and  in particular they are isometric. \end{proofC}

\begin{re}
Using the Galoisian terminology described in \cite[Section 2]{GOPRVS}, the pair
\[
(C(\frac{S^{15}}{K}),<\gamma>)
\]
is the minimal presentation of $M$.
\hfill\qed
\end{re}

\begin{re}
The fibers of $\pi$ in \hyperlink{te:C}{Theorem~C} inherit a 7-Sasakian structure (in the sense of~\cite{DeaNSM}) induced by the foliation $(VB)^\perp$ as in the \hyperlink{pr:proofB}{proof of Theorem~B}.
Indeed, this notion of 7-Sasakian structure on 15-dimensional spherical space forms seems to be the induced counterpart on the leaves of a canonical codimension one foliation on $M^{16}$. Note that, in accordance with~\cite{DeaNSM}, such a 7-Sasakian structure does not involve global vertical vector fields, but only a vertical foliation, whose transverse structure we have here related with the $\Spin{9}$-structure of $M^{16}$.\hfill\qed
\end{re}

The following is a different way of stating \hyperlink{te:C}{Theorem~C}.
\begin{co}\label{co:structure}
The set of isometry classes of locally, non globally, conformally parallel $\Spin{9}$ manifolds is in bijective correspondence with the set of triples
\[
\{(r,K,c_K)\st r\in\RR^+, K\leq\Spin{9} \text{ finite and free on }S^{15}, c_K\in\pi_0\left(N_{\Spin{9}}(K)\right)\},
\]
where $\pi_0$ stands for the connected component functor.
\hfill\qed
\end{co}

\begin{re}
We could also describe the map $\pi:M\rightarrow S^{1}$ in \hyperlink{te:C}{Theorem~C} as the Albanese map defined as follows. Fix any $x_0\in M$. For any $x\in M$ and any path $\gamma$ joining $x_0$ and $x$, define: 
\[
\alpha(x)\ug\left(\int_\gamma \theta \right)\mod G.
\]
Here $\theta$ is the Lee form of $M$, and $G\subset\RR$ is the additive subgroup of ``periods of $\theta$'', generated by the integrals $\int_{\sigma}\theta$ over the generators $\sigma$ of $H_1(M,\ZZ)$.
\hfill\qed
\end{re}

\section{Examples}

As a consequence of \hyperlink{te:B}{Theorem~B}, the examples will be in the context of the flat $\Spin{9}$-structure on $\RR^{16}$. Recalling the threefold approach to $\Spin{9}$-structures given by Definition \ref{de:spin9structure}, we refer to the data $V=V^9$, $\spinform{9}$, $\mathcal{R}$ and $\Spin{9}$ as the \emph{standard data}, and the standard inclusion $\SO{16}\subset\GL{16,\RR}$ can be viewed as equivalent to the choice of the standard basis $\{e_1,\dots,e_{16}\}$ of $\RR^{16}$ as orthonormal. Thus, another way to describe the flat $\Spin{9}$-structure on $\RR^{16}$ is the \emph{standard structure with respect to the standard basis} $\{e_1,\dots,e_{16}\}$.

Thus, if we choose a different basis $\mathcal{B}$ on $\RR^{16}$ that we declare to be orthonormal in a suitable metric $g_\mathcal{B}$, we can talk about the \emph{standard structure with respect to $\mathcal{B}$}. This means that we are choosing a different inclusion $i:\SO{16}\hookrightarrow\GL{16, \RR}$, but the structure is still standard in the sense that $V$, $\spinform{9}$ and $\mathcal{R}$ are induced by the standard ones using the inclusion $i$.

Observe that this holds even if the inclusion $i$ depends on the point $x\in\RR^{16}$, that is if $\mathcal{B}$ is not a basis on the vector space $\RR^{16}$, but a parallelization on the manifold $\RR^{16}$. In the same way, on any parallelizable $M^{16}$ with a fixed parallelization $\mathcal{B}$ one can speak of the \emph{standard $\Spin{9}$-structure on $M$ associated with $\mathcal{B}$}, whose associated objects will be denoted by $V_\mathcal{B}$, $\snform_\mathcal{B}$, $\mathcal{R}_\mathcal{B}$ and $g_\mathcal{B}$ (see~\cite{ParONS} and~\cite{ParHSS} for details).

\begin{ex}\label{ex:standard}
On $\RR^{16}\setminus 0$, consider the parallelization $\tilde{\mathcal{B}}\ug\{|x|\partial_1,\dots,|x|\partial_{16}\}$ where $\partial_1,\dots,\partial_{16}$ denotes the derivatives with respect to the standard coordinates. Look at the map
\[
p:\RR^{16}\setminus 0\longrightarrow S^{15}\times S^1, \; p(x)\ug(x/|x|,\log|x|\mod{2\pi}),
\]
and observe that $p$ projects $\tilde{\mathcal{B}}$ to a parallelization $\mathcal{B}\ug p_*(\tilde{\mathcal{B}})$ on $S^{15}\times S^1$ (see also~\cite[sections 6 and 7]{BruPEP} and~\cite{ParEPP}). Consider the standard $\Spin{9}$-structure $g_\mathcal{B}$ on $S^{15}\times S^1$ associated with $\mathcal{B}$. Then $g_\mathcal{B}$ is locally conformally parallel, since $p$ is a covering map, bundle-like by definition, so that $g_\mathcal{B}$ is locally given by $g_{\tilde{\mathcal{B}}}$, that is to say, by $|x|^{-2}g$, where $g$ is the flat metric on $\RR^{16}$.

As observed in \hyperlink{te:B}{Theorem~B}, the flat metric on $\RR^{16}\setminus 0$ is the cone metric on $C(S^{15})$. The metric $g_{\tilde{\mathcal{B}}}$ is instead the cylinder metric on the Riemannian universal covering of $S^{15}\times S^1$.
\hfill\qed
\end{ex}

\begin{re}\label{re:victor}
In~\cite{FriWSS} and~\cite{AgFGSV} the class of locally conformally parallel $\Spin{9}$-structures has been identified and studied, under the name of ``$\Spin{9}$-structures of vectorial type'' (cf.~the following Definition \ref{vectorial type}). We outline now a proof that, \emph{for $\Spin{9}$-structures, vectorial type is equivalent to locally conformally parallel.}

Following ~\cite{FriWSS} and~\cite{AgFGSV}, one can look at the splitting of the Levi-Civita connection in the principal bundle of orthonormal frames on $M$: 
\[
\nabla=\nabla^*\oplus\theta
\]
where $\nabla^*$ is the connection in the induced bundle of $\Spin{9}$-frames and $\theta$ is its orthogonal complement. Thus, $\theta$ is a 1-form with values in the ortogonal complement $\lie{m}$ defined by the splitting $\lieso{16}=\liespin{9}\oplus\lie{m}$ and, under canonical identifications, $\theta$ can be seen as a 1-form with values in $\Lambda^3(V)$.

Under the action of $\Spin{9}$, the space $\Lambda^1(M)\otimes\Lambda^3(V)$ decomposes as a direct sum of 4 irreducible components:
\[
\Lambda^1(M)\otimes\Lambda^3(V)=P_0\oplus P_1\oplus P_2\oplus P_3,
\]
and, looking at all the possible direct sums,  this yields 16 types of $\Spin{9}$-structures. The component $P_0$ identifies with $\Lambda^1(M)$. Thus:
\begin{de}\cite{AgFGSV}\label{vectorial type}
A $\Spin{9}$-structure is of \emph{vectorial type} if $\theta$ lives in $P_0$.
\end{de}

Now, let $(M,g)$ be a Riemannian manifold endowed with a $\Spin{9}$-structure of vectorial type. Let $\theta$ be as above, and let $\snform$ be its $\Spin{9}$-invariant 8-form. Now, $\theta=0$ implies that the holonomy of $M$ is contained in $\Spin{9}$ (cf.~\cite[page 21]{FriWSS}).

From~\cite[page 5]{AgFGSV} we know that the following relations hold:
\begin{equation}\label{eq:friedrich}
d\snform=\theta\wedge\snform,\qquad d\theta=0.
\end{equation}
Let $(M,\tilde{g})$ be the Riemannian universal cover of $(M,g)$ and let $\tilde{\snform}$, $\tilde{\theta}$ be the lifts of $\snform$, $\theta$ respectively. Then relations ~\eqref{eq:friedrich} hold as well for $\tilde{\snform}$ and $\tilde{\theta}$. Since $\tilde{M}$ is simply connected, then $\tilde{\theta}=df$, for some $f:\tilde{M}\rightarrow\RR$. Then, defining $g_0\ug e^{-f}\tilde{g}$ and $\snform_0\ug e^{-4f}\tilde{\snform}$, we have $d\snform_0=0$, that is the $\theta$-factor of $\snform_0$ is zero. Hence $g_0$ has holonomy contained in $\Spin{9}$, and on the other hand it is locally conformal to $g$. Thus $M$ can be covered by open subsets on which the metric is conformal to a metric with holonomy in $\Spin{9}$, which is Definition~\ref{de:lcp}.
\hfill\qed
\end{re}

\begin{re}
With the notations of \hyperlink{te:C}{Theorem~C}, the locally conformally parallel $\Spin{9}$-structure on $S^{15}\times S^1$ defined in Example~\ref{ex:standard} is associated with $K=\{\Id\}\subset\Spin{9}$. Since $N_{\Spin{9}}(K)=\Spin{9}$ is connected, there is only one locally conformally parallel $\Spin{9}$-structure on $S^{15}\times S^1$ (see Corollary~\ref{co:structure}). Thus, Example~\ref{ex:standard} is an alternative description of the $\Spin{9}$-structure of vectorial type  on $S^{15}\times S^1$ given in~\cite[example~2 at page~136]{FriWSS}, where the terminology ``$W_4$ type'' is used for vectorial type. See also~\cite{AgFGSV}.
\hfill\qed
\end{re}

\begin{ex}
According to Theorems~\hyperlink{te:B}{B} and~\hyperlink{te:C}{C}, to give examples of compact locally conformally parallel $\Spin{9}$ manifolds, one has to look at finite subgroups of $\Spin{9}$ acting without fixed points on $S^{15}$. The classification of such finite subgroups is not an easy problem, and we limit ourselves to exhibit some of them. They will show however that many finite quotients of $S^{15}$ may appear as fibers in the map of \hyperlink{te:C}{Theorem~C}. 

We describe in particular how $S^{15}$ is acted on ``diagonally'' and without fixed points by a subgroup $\Spdiag{1} \subset \Spin{9}$. 
Let $(x=h_1 +h_2 e, x'=h'_1+h'_2 e) \in \OO^2$ and define the following action of $q \in \Sp{1}$ on the \emph{first} octonionic coordinate $x=h_1 +h_2 e \in \OO$:
\[ 
A_q: \; h_1 +h_2 e \longrightarrow h_1 q + (\bar q h_2) e.
\]
Due to the identity $\overline{q_1 q_2} = \overline q_2 \overline q_1$, this is a right action $A_q: \OO \rightarrow \OO$ for each $q \in \Sp{1}$. In the real components of $x$, $A_q$ is represented by a matrix of $\SO{8}$, and indeed by a matrix in its diagonal $\SO{4} \times \SO{4}$ subgroup. 

Recall now the Triality Principle for $\SO{8}$. In the formulation we need here it can be stated as follows (cf. \cite[page 192] {GWZGHF} or \cite[pages 143-145]{DeSSTT}). 
\begin{triality}
Consider the triples $A,B,C \in \SO{8}$ such that for any $x,m \in \OO$:
\[
C(m)A(x)  = B(mx). 
\]
If any of $A,B,C$ is given, then the other two exist and are unique up to changing sign for both of them.
\end{triality}

Given $A \in \SO{8}$ we will call any of such matrices $\pm B, \pm C$ a \emph{triality companion} of $A$. 

Going back to the transformation $A_q \in \SO{8}$ defined by any $q \in \Sp{1}$, consider a pair $B_q, C_q \in \SO{8}$ of its triality companions. Thus, for any $x,m \in \OO$:
\[
C_q(m)A_q(x)  = B_q(mx),
\]
and define the following right action of $q \in \Sp{1}$ on $ \OO^2$:
\[
R_q: (x= h_1 +h_2 e, x'=h'_1+h'_2 e) \rightarrow (A_q x, B_q x').
\]
Thus, $R_q$ carries octonionic lines to octonionic lines, so that $R_q \in \Spin{9}$. In this way, 
a ``diagonal'' subgroup $\Spdiag{1} \subset \Spin{9}$ is defined, and $\Spdiag{1}$ is indeed a subgroup of the $\Spin{8} \subset \Spin{9}$ defined by triples $(A,B,C) \in \SO{8} \times \SO{8} \times \SO{8}$ obeying to the triality principle.

This action is without fixed points on $S^{15}$: from $A_q x  =  h_1q + \overline q h_2 e = h_1 + h_2 e, q\neq 1$ follows $h_1 = h_2 =0$, so that the only fixed points of $R_q$ could be on the unit sphere $S^7_\infty$ of the octonionic line $l_\infty$, on which we are acting by the triality companion $B_q$. Now, if $x' \in S^7_\infty$ is a fixed point of $B_q$, so is $-x'$ and then $B_q$ has to belong to a $\SO{7}$ subgroup of $\SO{8}$, rotating the equator of $S^7_\infty$ with respect to the poles $x'$ and $-x'$. But then the triple $(A_q,B_q,C_q)$ belongs to a $\Spin{7}$ subgroup of $\Spin{8}$ and hence any of $A_q,B_q,C_q$ has to belong to a $\SO{7} \subset \SO{8}$ (cf.~\cite[page 194] {MurES1}). Recall on the other hand that any subgroup $\SO{7} \subset \SO{8}$, when acting on the sphere $S^7$, admits a fixed point and it is conjugate with the standard $\SO{7}$ (cf.~\cite[Lemma 4, page 168]{VarSSS}). This is a contradiction , since $A_q$ has no fixed points.

We can now consider finite subgroups of $\Spdiag{1}$. Recall that any finite subgroup of $\Sp{1}$ is isomorphic to either a cyclic group or to the binary dihedral, tetrahedral, octahedral, or icosahedral group (see for instance \cite[Section 6.5]{CoxRCP}).
In the following, we associate a subgroup of $\Spdiag{1}$ with every group in the list of abstract finite subgroups of $\Sp{1}$.

\begin{itemize}
\item The cyclic group $\cyclic{n}=<a|a^n=1>$, for $n\ge 1$.
We can choose as generator $R_a$, where $a=e^{\frac{2\pi i}{n}}$.
\item The binary dihedral group $\bid{n}=<a,b|a^{2n}=1,b^2=a^n,b^{-1}ab=a^{-1}>$, for $n\ge 1$.
Choose here as generators $R_a,R_b$ with $a=e^{\frac{\pi i}{n}}$ and $b=j$. 
\item The binary tetrahedral group $\bit=<a,b,c|a^2=b^3=c^3=abc>$.
Choose now generators $R_a,R_b,R_c$ with $b=\frac{1+i+j+k}{2}$, $c=\frac{1+i+j-k}{2}$ and $a=bc$. 
\item The binary octahedral group $\bio=<a,b,c|a^2=b^3=c^4=abc>$.
Choose generators $R_a,R_b,R_c$ with $b=-\frac{1+i+j+k}{2}$, $c=\frac{1+i}{\sqrt{2}}$ and $a=bc$. 
\item The binary icosahedral group $\bii=<a,b,c|a^2=b^3=c^5=abc>$.
Let $\varphi\ug\frac{1+\sqrt{5}}{2}$ be the golden ratio. Choose generators $R_a,R_b,R_c$ where now $b=\frac{1+i+j+k}{2}$, $c=\frac{\varphi+\varphi^{-1}i+j}{2}$ and $a=bc$. 
\end{itemize}
Since any finite subgroup of $\Sp{1}$ is conjugate to one in the previous list, this classifies all locally conformally parallel $\Spin{9}$ manifolds such that $K=\ker \rho$ in \hyperlink{te:C}{Theorem~C} is contained in $\Spdiag{1}$.\hfill\qed
\end{ex}

\begin{re}
The Lee vector field on a locally conformally parallel $\Spin{9}$ manifold $M$ is never vanishing (see \hyperlink{pr:proofB}{proof of Theorem~B}). By~\cite[Proposition~1]{FriWSS} this means that $M$ admits a $\Spindelta{7}$-structure (in the sense of~\cite{FriWSS}). Thus, the classification of isometry types of $M$ reduces to the finding of finite subgroups of $\Spindelta{7}\subset\Spin{9}$ acting without fixed points on $S^{15}$.
\hfill\qed
\end{re}

{\textbf{Acknowledgements}} The authors thank Rosa Gini for her helpful support, 
and the two referees for their useful comments.

\end{document}